\documentclass[12pt,reqno]{amsart}

\usepackage[english]{babel}
\usepackage[utf8]{inputenc}
\usepackage{amsmath, amsthm, amssymb}
\usepackage{amsmath}
\usepackage{amsfonts}
\usepackage{amsthm}
\usepackage{mathrsfs}
\usepackage{amsrefs}
\usepackage{hyperref}

\usepackage{xcolor}

\usepackage[titletoc]{appendix}

\makeatletter	% to not show the acknowledgments in the Table of Contents
\ifcsname phantomsection\endcsname
\newcommand*{\qrr@gobblenexttocentry}[5]{}
\else
\newcommand*{\qrr@gobblenexttocentry}[4]{}
\fi
\newcommand*{\addsection}{%
	\addtocontents{toc}{\protect\qrr@gobblenexttocentry}%
	\section}
\makeatother

\theoremstyle{plain}
\newtheorem{theorem}{Theorem}[section]
\newtheorem{corollary}[theorem]{Corollary}

\newtheorem{proposition}[theorem]{Proposition}

\numberwithin{equation}{section}

\theoremstyle{definition}
\newtheorem{definition}[theorem]{Definition}

\newcommand{\R}{\mathbb{R}}
\newcommand{\e}{\varepsilon}
\newcommand{\abs}[1]{\left\lvert#1\right\rvert}

\begin{document}
	
	\title[Universal Hardy-Sobolev inequalities on hypersurfaces]
	{Universal Hardy-Sobolev inequalities on hypersurfaces of Euclidean space}
	
	\author[X. Cabr\'e]{Xavier Cabr\'e}
	\author[P. Miraglio]{Pietro Miraglio}
	
\address{X.C.\textsuperscript{1,2,3}
	\textsuperscript{1} ICREA, Pg.\ Lluis Companys 23, 08010 Barcelona, Spain
	\&
	\textsuperscript{2} Universitat Polit\`ecnica de Catalunya, Departament de Matem\`{a}tiques, Diagonal 647, 08028 Barcelona, Spain
	\&
	\textsuperscript{3} BGSMath, Campus de Bellaterra, Edifici C, 08193 Bellaterra, Spain.
}
\email{xavier.cabre@upc.edu}
	
	\address{P.M.\textsuperscript{1,2},
		%\newline 
		\textsuperscript{1} Universit\`a degli Studi di Milano, Dipartimento di Matematica, via Cesare Saldini 50, 20133 Milano, Italy
		%\newline
		\textsuperscript{2} Universitat Polit\`ecnica de Catalunya, Departament
		de Matem\`{a}tiques, Diagonal 647,
		08028 Barcelona, Spain}
	\email{pietro.miraglio@upc.edu}
	
	
	\thanks{X.C. and P.M. are supported by the MINECO grant MTM2017-84214-C2-1-P and are members of the Catalan research group 2017SGR1392.
	}
	
	\begin{abstract}
		In this paper we study Hardy-Sobolev inequalities on hypersurfaces of~$\mathbb{R}^{n+1}$, all of them 
		involving a mean curvature term and having universal constants independent of the hypersurface.
		We first consider the celebrated Sobolev inequality of Michael-Simon and Allard,
		in our codimension one framework. Using their ideas, but simplifying their presentations, we give
		a quick and easy-to-read proof of the inequality. Next, we establish two new Hardy inequalities on hypersurfaces. 
		One of them originates from an application to the regularity theory of stable solutions to semilinear elliptic equations.
		The other one, which we prove by exploiting a ``ground 
		state'' substitution, improves the Hardy inequality of 
		Carron. With this same method, we also obtain an improved Hardy or Hardy-Poincaré inequality.
	\end{abstract}
	
	\maketitle
	
	\tableofcontents

\section{Introduction}
In this article we establish some new Hardy inequalities on hypersurfaces of Euclidean space. 
As the one of Carron~\cite{Car} --- for which we find an improved version --- all of them involve
a mean curvature term 
and have universal constants. Our inequalities have 
their origin in the recent work \cite{C1} by the first author on the regularity theory of 
stable solutions to semilinear elliptic equations. The paper \cite{C1} established the regularity of such 
solutions up to dimension four, for all nonlinearities, by using a foliated version of one of our new Hardy inequalities --- the one of
Theorem~\ref{thm_hardy_ibp_p2} below. In this way, \cite{C1} succeeded to greatly simplify the 2010
proof of the same result found in~\cite{C} by the first author\footnote{
In the case of nonnegative nonlinearities, regularity of stable solutions up to the optimal dimension nine has been recently obtained by Figalli, Ros-Oton, Serra, and the first author~\cite{CFRS}.
}. In addition, \cite{C} used the Michael-Simon
and Allard Sobolev inequality, which is a more sophisticated tool than our Hardy inequality. In fact,
one of the features of the current paper is that proofs are rather elementary --- even if they concern
functions defined on hypersurfaces. In particular, in Section \ref{section_ms&a} we give
a quick and easy-to-read proof of the Sobolev inequality of Michael-Simon and Allard, for completeness and since we believe it can be useful for potential readers.

Let us start presenting the inequality of Michael-Simon and Allard. 
In 1967, Miranda~\cite{Miranda} established that the Sobolev inequality holds in its Euclidean form, but 
possibly with a different constant, on every minimal hypersurface of $\R^n$.
Some years later, a more general Sobolev inequality for $k$-submanifolds of $\R^n$, not necessarily minimal, was proved independently by Michael and Simon~\cite{MS} and by Allard~\cite{A}. This inequality was subsequently generalized by Hoffman and Spruck~\cite{HS} to submanifolds of general Riemannian manifolds.

In the context of hypersurfaces of $\R^{n+1}$, i.e., submanifolds of the Euclidean space with codimension one, the Sobolev inequality reads as follows.

\begin{theorem}[Allard~\cite{A}, Michael-Simon~\cite{MS}]
	\label{thm_MS_functions} 
	Let $M$ be a smooth $n$-dimensional hypersurface of $\R^{n+1}$, $p\in[1,n)$, and $\varphi\in C^1(M)$ have compact support in $M$. 
	If $M$ is compact without boundary, any function $\varphi\in C^1(M)$ is allowed.
	
	Then, there exists a positive constant $C$, depending only on $n$ and $p$, such that
	\begin{equation}\label{ms&a}
	\lVert \varphi\rVert_{L^{p^*}(M)}^p\leq C\int_M\big(\abs{\nabla_T\varphi}^p+\abs{H\varphi}^p\big)\,dV,
	\end{equation}
	where $ p^*=np/(n-p) $ is the Sobolev exponent, $H$ is the mean curvature of $M$, and $\nabla_T$ denotes the tangential gradient to $M$.
\end{theorem}

The constant $C$ in~\eqref{ms&a} is universal, in the sense that it depends only on the dimension $n$ and on the exponent $p$, but not on $M$. Thus, the geometry of the hypersurface plays a role just through the term involving the mean curvature $H$ appearing in the right-hand side of \eqref{ms&a}. In particular, when~$M$ is 
minimal\footnote{
Here and throughout the paper, \textit{minimal} hypersurface refers to a hypersurface which is a critical point (not necessarily a minimizer) of the area functional, i.e., a hypersurface with zero mean curvature.}, such term vanishes and we recover the Sobolev inequality proved earlier by Miranda~\cite{Miranda}.

The formulation of the Michael-Simon and Allard inequality stated in Theorem~\ref{thm_MS_functions} can be easily deduced, using standard tools, from the following isoperimetric inequality.

\begin{theorem}[Allard~\cite{A}, Michael-Simon~\cite{MS}]
	\label{thm_MS&A}
	Let $M$ be a smooth $n$-dimensional hypersurface of $\R^{n+1}$ and $E\subset M$ a smooth domain with compact closure in $M$. Then	
	\begin{equation}
	\label{MSineq}
	\lvert E \rvert^{\frac{n-1}{n}}\leq C\Big(\textnormal{Per}(E)+\int_E\lvert H \rvert\,dV\Big),
	\end{equation}
	where $H$ is the mean curvature of $M$, $\textnormal{Per}(E)$ is the perimeter of $E$, and $C$ is a constant depending only on the dimension $n$ of $M$. 
\end{theorem}

The inequalities presented in Theorems \ref{thm_MS_functions} and \ref{thm_MS&A} were proven in the seventies 
in~\cites{A,MS}, in independent works. In~\cite{A} the proof is based on establishing an isoperimetric inequality, 
like the one in Theorem~\ref{thm_MS&A}, for $k$-dimensional varifolds of~$\R^n$. From it, Theorem 
\ref{thm_MS_functions} can be easily deduced. Instead, in~\cite{MS} the authors prove directly a 
Sobolev inequality for submanifolds of $\R^n$ of any codimension. A slight modification of the argument 
in~\cite{MS}, due to Leon Simon, is presented in the monograph~\cite[Theorem~3.11]{ColMin}.

In the current paper, where we focus on the case of hypersurfaces of $\R^{n+1}$, we first 
present a quick and easy-to-read proof of the Michael-Simon and Allard inequality.
Our proof uses mainly the tools of Michael and Simon~\cite{MS} but contains two simplifications:
we target at the isoperimetric inequality (instead, \cite{MS} pursues the Sobolev inequality) and we
use a quick Gronwall-type argument from Allard~\cite{A}. 

After \cites{A,MS}, alternative proofs of the Sobolev inequality have been found. In the case of two-dimensional
minimal surfaces (with any codimension), Leon Simon gave a rather simple proof which, 
in addition, carries a constant optimal up to a factor of 2. This work remained unpublished, but is presented in~\cite{Choe,Top}. An improved version of it, which holds in any two-dimensional surface, not necessarily minimal,
was found by Topping~\cite{Top}. 
In the case of submanifolds of arbitrary dimension and codimension,
Castillon~\cite{Cas} gave a new proof of
the Michael-Simon  and Allard Sobolev inequality by using optimal transport methods. 
Finally, an important result has been obtained very recently by Brendle~\cite{Brendle}, also in the case of arbitrary dimension and codimension. He finds a new
proof of the Sobolev inequality that, in addition, carries the sharp constant in the case of
minimal submanifolds of~$\R^{n+1}$ of codimension at most two. 
This is the first time that the Michael-Simon and Allard inequality is proved in minimal submanifolds (or even minimal hypersurfaces) with the optimal Euclidean constant.
Brendle's method is a clever extension of the 
proof of the sharp Euclidean isoperimetric inequality found by the first author in~\cite{C08}. 
In Appendix~\ref{appendix_constant} we describe it in some more detail, together with 
other results about optimal constants in the Michael-Simon and Allard inequality --- a topic that has been studied
mainly in the case of submanifolds being either minimal or compact without boundary.

Our interest in the Michael-Simon and Allard inequality originates from an application of it to the regularity theory for semilinear elliptic equations.
More precisely, in 2010 the first author proved in~\cite{C} an a priori estimate for stable solutions to
$-\Delta u= f(u)$ in bounded domains of $\R^{n+1}$,
using as a key tool the Michael-Simon and Allard inequality~\eqref{ms&a} applied on every level set of $u$. 
The estimate in~\cite{C}, whose proof was quite delicate, led to the regularity of stable solutions 
in dimensions $n+1\leq4$ for every smooth nonlinearity $f$.

An alternative and much simpler proof of this same result has been recently found by the first author~\cite{C1}. This new method does not use the Michael-Simon and Allard inequality, but it is based instead on a new Hardy inequality with 
sharp constant --- also established in~\cite{C1} --- adapted to the level sets of a function $u$. In~\cite{C1}, this Hardy inequality is later used with $u$ being a stable solution to $-\Delta u= f(u)$ in a bounded domain $\Omega\subset\R^{n+1}$. To describe the new inequality, for every smooth function~$u$ we consider its radial derivative $u_r=\nabla u\cdot x/\abs{x}$. Then, for every $\varphi\in C_c^1(\Omega)$, with $\Omega\subset\R^{n+1}$ an open set, and every parameter $a\in[0,n)$, the Hardy inequality from~\cite{C1} states that
\begin{equation}\label{hardy_cabre}
\begin{split}
&(n-a)\int_{\Omega}\abs{\nabla u}\frac{\varphi^2}{\abs{x}^a}\,dx+a\int_{\Omega}\frac{u_{r}^2}{\abs{\nabla u}}\,\frac{\varphi^2}{\abs{x}^a}\,dx
\\
&\hspace{1.5cm}\leq
\left(\int_{\Omega}\abs{\nabla u}\frac{\varphi^2}{\abs{x}^a}\,dx\right)^\frac12
\left(\int_{\Omega}\abs{\nabla u} \frac{4\abs{\nabla_T\varphi}^2+\abs{H\varphi}^2}{\abs{x}^{a-2}}\,dx\right)^\frac12,
\end{split}
\end{equation}
where the tangential gradient $\nabla_T$ and the mean curvature $H$ are referred to the level sets\footnote{
	By Sard's theorem, if $u\in C^\infty$, almost every level set of $u$ is a smooth embedded hypersurface of $\R^{n+1}$.
} of $u$. 

Throughout the paper, the mean curvature $H$ is the sum, and not the arithmetic mean, of the principal curvatures. Therefore, when $M$ is the $n$-dimensional unit sphere, we have $H=n$.

Using the coarea formula, from \eqref{hardy_cabre} one can deduce the following Hardy inequality
on a single hypersurface\footnote{
For this, one applies~\eqref{hardy_cabre} with $u(x)={\rm dist }(x,M)$ in $\Omega_\varepsilon:=\{ 0<u<\e\}\cap B_R$ 
after extending $\varphi\in C^1_c(M\cap B_R)$ to be constant in the normal directions to $M$.
Then one divides the inequality by $\varepsilon$ and lets $\varepsilon\to 0$. This requires a more general version
of \eqref{hardy_cabre} in which the part of $\partial\Omega=\partial\Omega_\varepsilon$ where $\varphi\neq 0$
is divided into two open subsets with $u$ being constant on each of them (equals $0$ and $\e$ in our case).
This version of \eqref{hardy_cabre} can be proved exactly as in \cite{C1},
after checking that the foliated integration by parts formula of
Lemma 2.1 in \cite{C1} also holds for these boundary conditions.
}
$M$.
Here and throughout the paper, $C^1_c(M)$ denotes the space of~$C^1$ functions with 
compact support on~$M$. In case $M$ is a compact hypersurface without boundary, then $C^1_c(M)=C^1(M)$.

\begin{theorem}
	\label{thm_hardy_ibp_p2}
	Let $M$ be a smooth hypersurface of $\R^{n+1}$ and $a\in[0,n)$. Then, for every $\varphi\in C^1_c(M) $ we have
	\begin{equation}\label{hardy_ibp_p2}
	\begin{split}
	&(n-a)\int_M\frac{\varphi^2}{\abs{x}^a}\,dV+a\int_M\left(\frac{x}{\abs{x}}\cdot\nu_M\right)^2\frac{\varphi^2}{\abs{x}^a}\,dV
	\\ 
	&\hspace{3cm}\leq 
	\left(\int_M\frac{\varphi^2}{\abs{x}^a}\,dV\right)^\frac12
	\left(\int_M\frac{4\abs{\nabla_T\varphi}^2+\abs{H\varphi}^2}{\abs{x}^{a-2}}\,dV\right)^\frac12,
	\end{split}
	\end{equation}
	where $\nu_M$ is the unit normal to $M$ in $\R^{n+1}$.
\end{theorem}

In this paper we present a direct proof of Theorem \ref{thm_hardy_ibp_p2}
which does not rely on the more involved proof in~\cite{C1} of its foliated version \eqref{hardy_cabre}. 
Then, using the coarea formula, 
we deduce~\eqref{hardy_cabre} from \eqref{hardy_ibp_p2} --- see Corollary~\ref{cor_hardy} and its proof. 
Moreover, in Theorem~\ref{thm_hardy_ibp} we give a version of~\eqref{hardy_ibp_p2}, and 
thus of~\eqref{hardy_cabre}, for an arbitrary exponent $p\geq1$ instead of $p=2$. 
Our proof of Theorem~\ref{thm_hardy_ibp_p2} is elementary and based on the use of 
the tangential derivatives~$\delta_i$, which we recall in Appendix~\ref{appendix_giusti}.

Note that when $M=\R^n$, $n\geq3$, and $a=2$, then \eqref{hardy_ibp_p2} 
is the Euclidean Hardy inequality with best constant, 
\begin{equation}
\label{flat_hardy}
\frac{\left(n-2\right)^2}{4}\int_{\R^n}\frac{\varphi^2}{\abs{x}^2}\,dx
\leq
\int_{\R^n} \abs{\nabla \varphi}^2\,dx,
\end{equation}
since the second term in the left-hand side of \eqref{hardy_ibp_p2} vanishes.
Instead, when $M$ is close to a sphere in $\R^{n+1}$ centered at the origin, such term becomes important
and could even make larger the constant $n-a$ in the first 
term in the left-hand side of~\eqref{hardy_ibp_p2}. This is one of the interesting points of  
our result. Note, however, that \eqref{hardy_ibp_p2} is trivial when $M=S^n$, since $H\equiv n$.

The foliated version \eqref{hardy_cabre} of our Hardy inequality  was used in \cite{C1} to establish 
the boundedness of stable solutions to semilinear elliptic equations up to dimension $n+1 \leq 4$ for all nonlinearities. 
Thanks to our improved version, which includes the second term on its left-hand side, the same proof gave, in the radial case, regularity up to the optimal dimension $n+1\leq9$ ---
since one has $u_r^2=\abs{\nabla u}^2$ in its left hand side for radial solutions.
In the nonradial case, the optimal result in dimension $n+1\leq9$ has been recently obtained, for nonnegative nonlinearities, by 
Figalli, Ros-Oton, Serra, and the first author~\cite{CFRS}.
This result, whose proof does not rely on Hardy-Sobolev inequalities, gives a complete answer to a 
long standing open question posed by Brezis~\cite{B} and by Brezis and V\'azquez~\cite{BV}.

The application of inequality \eqref{hardy_cabre} to the regularity theory of stable solutions has 
been extended by the second author in \cite{Mi} to nonlinear equations involving the $p$-Laplacian. 
It is worth pointing out here that this is done using the quadratic version \eqref{hardy_cabre} of 
the Hardy inequality on the level sets, and not the one for a general exponent $p$ stated in 
Corollary \ref{cor_hardy}.

A related but different Hardy inequality on hypersurfaces of $\R^{n+1}$ was proved in 1997 by Carron~\cite{Car}. 
It states that in every dimension $n\geq3$ and for all $\varphi\in C^1_c(M) $ it holds that
\begin{equation}\label{carron}
\frac{(n-2)^2}{4}\int_M\frac{\varphi^2}{\abs{x}^2}\,dV 
\leq 
\int_M \bigg(\abs{\nabla_T\varphi}^2+ \frac{n-2}{2}\,\frac{\abs{H}\varphi^2}{\abs{x}}\bigg)\,dV.
\end{equation}
In particular, this established that the Hardy inequality in its Euclidean form and with its best constant 
holds in every minimal hypersurface of $\R^{n+1}$. Observe that this also follows from our 
Theorem~\ref{thm_hardy_ibp_p2} by taking $a=2$. Also in the context of minimal hypersurfaces, 
in Section~\ref{section_hardy} we will prove an analogue sharp Hardy inequality with exponent $p\neq2$, 
namely, \eqref{p_hardy_flat}. Even if not explicitly mentioned in~\cite{KO1},
inequality~\eqref{p_hardy_flat} also follows by the results of 
Kombe and Özaydin~\cite[Theorem 2.1]{KO1}.\footnote{
One uses \cite[Theorem 2.1]{KO1} with $\alpha=0$ and $\rho=\abs{x}$, together with the well-known inequality
$\Delta \rho \geq (n-1)/\rho$ involving the Laplace-Beltrami operator, 
which holds if $H\equiv 0$ as we show in the beginning of subsection~3.2. 
}

In~\cite{Car} Carron proved also an intrinsic Hardy inequality on Cartan-Hadamard manifolds. 
His work gave rise to numerous papers in the topic of Hardy inequalities on manifolds,
some of which are commented on next.
Carron's work was extended to general Riemannian manifolds by Kombe and Özaydin~\cites{KO1,KO2}, 
who also included the case of a general exponent~$p$ instead of only $p=2$.
Some intrinsic Hardy inequalities 
with general weights, not necessarily of the power type, are studied by D'Ambrosio and Dipierro~\cite{DaDi}.
The case of the hyperbolic space~$\mathbb{H}^n$ and related manifolds is treated by
Berchio, Ganguly, Grillo, and Pinchover~\cites{BGG, BGGP}, obtaining sharp constants and improved 
versions of the inequality.
Finally, let us mention the recent work of Batista, Mirandola, and Vit\'orio~\cite{BMV1} 
improving Carron's inequality with power weights
in the setting of manifolds isometrically immersed in Cartan-Hadamard manifolds.

In Theorem~\ref{thm_hardy_gs_p2} below, we obtain an improved version of Carron's inequality~\eqref{carron} in the case of 
hypersurfaces of $\R^{n+1}$ by adding a nonnegative term on its left-hand side (the same term as in the inequality of Theorem \ref{thm_hardy_ibp_p2} with $a=2$).
We could not find such additional term within the literature on Hardy's inequalities.
In addition, our method of proof towards Hardy's inequalities 
is different from the ones in~\cites{BMV1,Car}, for instance.

\begin{theorem}
	\label{thm_hardy_gs_p2}
	Let $M$ be a smooth hypersurface of $\R^{n+1}$ with $ n\geq3 $. Then, for every $\varphi\in C^1_c(M) $ we have
	\begin{equation}\label{hardy_car}
	\begin{split}
	&\hspace{-0.5cm}\frac{(n-2)^2}{4}\int_M\frac{\varphi^2}{\abs{x}^2}\,dV 
	+
	\frac{n^2-4}{4}\int_M\left(\frac{x}{\abs{x}}\cdot\nu_M\right)^2\frac{\varphi^2}{\abs{x}^2}\,dV
	\\
	&\hspace{5cm}\leq 
	\int_M \bigg(\abs{\nabla_T\varphi}^2+ \frac{n-2}{2}\,\frac{\abs H\varphi^2}{\abs{x}}\bigg)\,dV,
	\end{split}
	\end{equation}
	where $\nu_M$ is the unit normal to $M$ in $\R^{n+1}$.
\end{theorem}
As in Theorem \ref{thm_hardy_ibp_p2}, the second term in the left-hand side of \eqref{hardy_car} is of special interest when $M$ is close to be a sphere of $\R^{n+1}$ centered at the origin.

We prove Theorem~\ref{thm_hardy_gs_p2} using a technique which, in the case of the Euclidean space, is known as ground state substitution. It dates back at least to the time of Jacobi and it has been applied for instance in the spectral theory of Laplace and Schrödinger operators. It is based on writing the function $\varphi$ as $\varphi=v\omega $, where typically $ \omega $ is a positive solution of the Euler-Lagrange equation of the energy functional associated with the inequality.
This method has been used in the Euclidean setting by Brezis and V\'azquez~\cite{BV} to obtain an improved Hardy inequality in $\R^n$, stated in~\eqref{bv_impr_hardy} below. 
The ground state substitution is essentially equivalent to the use of a Picone identity, as done in 
Abdellaoui, Colorado, and Peral~\cite{ACP}, where the authors also obtained some improved Hardy 
inequalities in domains of $\R^n$.
More recently, Frank and Seiringer~\cite{FS} used the ground state substitution to prove fractional Hardy inequalities in~$\R^n$. 
We will use this method in the framework of functions defined on a hypersurface of the Euclidean space --- something that we could not find in previous literature. In our proof we will take $\omega(x)=\abs{x}^{-(n-2)/2}$.

The two inequalities of Hardy type in Theorems~\ref{thm_hardy_ibp_p2} and~\ref{thm_hardy_gs_p2} are different in their formulations and independent in their proofs. Their statements differ mainly in the mean curvature term, containing $H^2$ versus $\abs{H}/\abs{x}$, respectively. At the same time, their proofs use distinct techniques. In addition, our proof of Theorem~\ref{thm_hardy_ibp_p2} works for an arbitrary exponent $p\geq1$ ---  see Theorem~\ref{thm_hardy_ibp} for the general statement --- while the one of Theorem \ref{thm_hardy_gs_p2} gives a significant result only in the case $p=2$.
Indeed, with our technique one can prove a $p$-version of \eqref{hardy_car}, but it is of less interest due to the presence of the second fundamental form in its right-hand side (instead of only the mean curvature). Moreover, its left-hand side contains some factors $(\abs{x_T}/\abs{x})^{p-2}$, where $x_T$ is the tangential part of the position vector $x$.

As a simple interpolation of the Michael-Simon and Allard inequality and of Theorem~\ref{thm_hardy_ibp_p2} with $a=2$, we obtain the following Hardy-Sobolev inequality on hypersurfaces of $ \R^{n+1} $.

\begin{corollary}\label{cor_ckn_p2}
	Let $M$ be a smooth hypersurface of $\R^{n+1}$ with $n\geq3$, $ b\in[0,1] $, and $\varphi\in C^1_c(M)$. Then, there exists a positive constant $C$ depending only on the dimension $n$, such that
	\begin{equation}\label{h_s_p2}
	\left(\int_M\frac{\abs{\varphi}^{\frac{2(n-2b)}{n-2}}}{\abs{x}^{2b}}\,dV\right)^\frac{n-2}{n-2b}\leq C\int_M\big(\abs{\nabla_T\varphi}^2+\abs{H\varphi}^2\big)\,dV.
	\end{equation}
\end{corollary}

Corollary~\ref{cor_ckn}, which is the general version of \eqref{h_s_p2} with exponents $ p\in[1,n) $, covers some possible choices of the parameters in Caffarelli-Kohn-Nirenberg type inequalities on hypersurfaces. Indeed, in~\cite{BMV}, Batista, Mirandola, and Vit\'orio prove a Caffarelli-Kohn-Nirenberg inequality for submanifolds of Riemannian manifolds, from which Corollary~\ref{cor_ckn} can be deduced, perhaps with a different constant.
However, the proof in~\cite{BMV} is delicate and relies on Riemannian geometry techniques, while we easily show Corollary~\ref{cor_ckn} as an interpolation of our previous results in the setting of hypersurfaces of~$\R^{n+1}$.

The classical Hardy's inequality has been improved in the Euclidean setting in many ways, see 
for instance~\cites{ACP, DD, ACR, BV, VZ, BFT}.
Many of these improvements consist of adding a positive term on the left-hand side of the inequality. 
This additional term has to be of lower order than the Hardy integral, by the optimality of the constant $(n-2)^2/4$.
This is done for example by Brezis and V\'azquez in~\cite[Theorem~4.1]{BV}, where they get an improvement in the Poincaré sense. Namely, they control both a Hardy-type integral and the $ L^2 $-norm of a function in terms of the $ L^2 $-norm of its gradient. For any bounded domain $ \Omega \subset \R^n $, any dimension $ n\geq2 $ and for every function $\varphi\in H_0^1(\Omega) $, their result states that
\begin{equation}\label{bv_impr_hardy}
\frac{\left(n-2\right)^2}{4}\int_\Omega\frac{\varphi^2}{\abs{x}^2}\,dx+ H_2\left(\frac{\omega_n}{\abs{\Omega}}\right)^\frac2n\int_\Omega\varphi^2\,dx
\leq
\int_\Omega \abs{\nabla \varphi}^2\,dx,
\end{equation}
where $H_2$ is the first eigenvalue of the Laplacian in the unit ball of $\R^2$, 
hence positive and independent of $n$, and $\omega_n$ is the measure of the unit ball in $\R^n$.

Using the ground state substitution as in the proof of Theorem~\ref{thm_hardy_gs_p2}, we prove the following 
analogue of the improved Hardy inequality by Brezis and V\'azquez, now on hypersurfaces of $\R^{n+1}$. 
We require functions to have compact support on the hypersurface $M$ intersected with a ball of radius~$r$
in the ambient space.
\begin{theorem}
	\label{thm_hardy+_p2}
	Let $M$ be a smooth hypersurface of $\R^{n+1}$ with $ n\geq2 $, and $B_r=B_r(0)\subset\R^{n+1}$ be the $(n+1)$-dimensional open ball of radius $r$ centered at the origin. 
	
	Then, for every $\varphi\in C^1_c(B_r\cap M) $ we have
	\begin{equation}
	\label{imprhardy_p2}
	\begin{split}
	&\hspace{-0.1cm}\frac{(n-2)^2}{4}\int_M\frac{\varphi^2}{\abs{x}^2}\,dV
	+
	\frac{n^2-4}{4}\int_M\left(\frac{x}{\abs{x}}\cdot\nu_M\right)^2\frac{\varphi^2}{\abs{x}^2}\,dV
	+ \frac{1}{2r^2}\int_M \varphi^2\,dV
	\\ 
	&\hspace{3.7cm}\leq
	\int_M \left(\lvert\nabla_T\varphi\rvert^2 + \frac{n-2}{2}\,\frac{\abs{H}\varphi^2}{\abs x}+\frac14\abs{H\varphi}^2\right)\,dV,
	\end{split}
	\end{equation}
	where $\nu_M$ is the unit normal to $M$ in $\R^{n+1}$.
\end{theorem}

The proof of this result combines the one of Theorem~\ref{thm_hardy_gs_p2} (which uses the ground state substitution) with a Poincaré inequality in hypersurfaces of $\R^{n+1}$, stated in Proposition~\ref{prop_poincare}. The former argument brings the first mean curvature term in~\eqref{imprhardy_p2}, while the latter brings the second one.
Note that these are the same curvature terms that appear in~\eqref{hardy_car} and~\eqref{hardy_ibp_p2}.

\subsection{Structure of the paper} In Section~\ref{section_ms&a} we give a quick and easy-to-read proof of the 
Michael-Simon and Allard inequality. In Section~\ref{section_hardy} we prove the Hardy inequalities stated in 
Theorems~\ref{thm_hardy_ibp_p2} and~\ref{thm_hardy_gs_p2}. Finally, Section~\ref{section_mixed} deals with the Hardy-Sobolev inequality of Corollary~\ref{cor_ckn_p2} and the improved Hardy-Poincaré inequality of Theorem~\ref{thm_hardy+_p2}.
The appendices concern tangential derivatives and divergence theorems on hypersurfaces, as well as optimal constants in the Michael-Simon and Allard inequality.

\section{The Michael-Simon and Allard inequality}
\label{section_ms&a}

In this section we present a proof of the Michael-Simon and Allard inequality on hypersurfaces of $\R^{n+1}$ stated in Theorem~\ref{thm_MS&A}. This result is a generalization of the isoperimetric inequality on minimal surfaces of Miranda~\cite{Miranda} and it is due to Michael and Simon~\cite{MS} and independently to Allard~\cite{A}. 
Throughout the paper, $M$ is an $n$-dimensional smooth hypersurface of $\R^{n+1}$ with mean curvature $H$, 
while $E$ is a bounded subset of $M$ with $n$-dimensional Hausdorff measure $\lvert E\rvert$ and perimeter $\text{Per}(E)$. 

In the proof of Theorem~\ref{thm_MS&A}, the notions of tangential derivatives and tangential divergence are crucial. We introduce them in Definition~\ref{def_div}, following the book of Giusti~\cite{G}.
We also use the following divergence formula on $M$ --- 
see \eqref{app_divthm} in Appendix \ref{appendix_giusti} for details. If $Z$ is a \textit{tangent} vector field 
on $M$, $\Omega$ a smooth domain in~$M$, $\textnormal{div}_T Z$ the tangential divergence with respect to the hypersurface $M$, and $\nu_\Omega$ is the outer normal vector along $\partial\Omega$ to $\Omega$, then 
\begin{equation}
\label{divthm}
\int_\Omega \textnormal{div}_T Z\,dV=\int_{\partial\Omega}Z\cdot \nu_\Omega\,dA.
\end{equation}

In the proof of Theorem~\ref{thm_MS&A}, we apply \eqref{divthm} in the domain $E_\rho=E\cap B_\rho(y)$, 
where $B_\rho(y)$ is the ball of $\R^{n+1}$ with radius $\rho$ and center $y\in E$. In general, 
the boundary of $E_\rho$ is not smooth. However, applying Sard's theorem on $\partial E$ to the 
function ``distance to~$y$'' defined on $\partial E$, we deduce that almost all its values are 
regular on~$\partial E$. Now, for these regular values $\rho$, if the hypersurfaces $\partial E$ and 
$\partial B_\rho(y)$ intersect each other, then they do it transversally and, as a consequence,
the boundary of $ E_\rho $ is Lipschitz. Recall that this will happen for almost every 
$ \rho>0 $. At the same time, it is possible to state~\eqref{divthm} for a domain $ \Omega $ with 
Lipschitz boundary, approximating it with a sequence of smooth sets.

By computing the tangential divergence of the position vector $x$, we can deduce an important equality which is the starting point of the proof of Theorem~\ref{thm_MS&A}:
\begin{equation}
\label{div2}
\textnormal{div}_T x=\sum_{i=1}^{n+1}\delta_i x^i
=\sum_{i=1}^{n+1}\left(\partial_i x^i-\nu_M^i\sum_{j=1}^{n+1}(\partial_j x^i)\nu_M^j\right)
=n+1-\sum_{i=1}^{n+1}(\nu_M^i)^2=n,
\end{equation}
where $\delta_i$ for $i=1,\dots,n+1$ denote the tangential derivatives defined in Appendix~\ref{appendix_giusti}.
Before starting the proof of Theorem~\ref{thm_MS&A}, we also recall that \[ H=\text{div}_T\nu_M, \] where 
$ \nu_M $ is the normal vector to $M$ --- not to be confused with $\nu_\Omega$ in \eqref{divthm} ---, 
and that the mean curvature vector is $\mathcal{H}=H\nu_M.$

\begin{proof}[Proof of Theorem~\ref{thm_MS&A}]
	Let $y\in E$ and define $E_\rho:=E\cap B_\rho(y)$, where $B_\rho(y)$ is the ball of $\R^{n+1}$ centered at $y$ of radius $\rho>0$. We start the proof by showing the validity for almost every $ \rho>0 $ of the inequality
	\begin{equation}\label{lemma1}
	n\lvert E_\rho\rvert \leq \rho \Big(\text{Per}(E_\rho)+\int_{E_\rho}\lvert H \rvert\,dV\Big).
	\end{equation}
	To prove it, for simplicity and without loss of generality, we may take $y=0$. We denote by $\nu_{E_\rho}$ the outer normal vector along $\partial E_\rho$ to $E_\rho$. We call $ x_T $ the tangential part of the position vector $x$ with respect to the hypersurface $M$ and thus, using~\eqref{div1}, we have
	\begin{equation*}
	\begin{split}
	\textnormal{div}_T x&= \textnormal{div}_T\big(x_T+(x\cdot\nu_M)\nu_M\big)
	\\&=\textnormal{div}_T x_T + \nabla_T(x\cdot\nu_M)\cdot\nu_M + (x\cdot\nu_M) \textnormal{div}_T \nu_M
	\\&=\textnormal{div}_T x_T + (x\cdot\nu_M) H.
	\end{split}
	\end{equation*}
	Now, as pointed out after \eqref{divthm}, the boundary of $ E_\rho $ is Lipschitz for almost every 
        $ \rho>0 $. Hence, for such values of $\rho$ we can integrate  in $ E_\rho $ the last equality, 
        and using \eqref{div2} and \eqref{divthm}, we deduce
	\begin{equation*}
	\begin{split}
	n\lvert E_\rho\rvert&= \int_{E_\rho} \textnormal{div}_T x\,dV = \int_{\partial E_\rho}x_T\cdot \nu_{E_\rho}\,dA + \int_{E_\rho}(x\cdot \nu_M)H \,dV
	\\ &\leq \rho\,\text{Per}(E_\rho)+\rho\int_{E_\rho}\lvert H \rvert\,dV,
	\end{split}
	\end{equation*}
	proving \eqref{lemma1}.
	
	Back to a general point $y\in E$, for the regular values $\rho$ corresponding to the point~$y$
	--- as defined after \eqref{divthm} ---, we have that if the smooth hypersurfaces $\partial E$ and 
        $\partial B_\rho(y)$ intersect each other, then they do it transversally. As a consequence, we deduce that\footnote{ Here we use the coarea formula, which gives $\frac{d}{d\rho}|E_\rho|=\int_{\partial B_\rho(y) \cap E} |\nabla_T|x-y||^{-1} dV$.}
	\begin{equation*}
	\begin{split}
	\text{Per}(E_\rho)&=|\partial B_\rho(y) \cap E|+|\partial E \cap B_\rho(y)|\\
	&\leq \frac{d}{d\rho}|E_\rho|+|\partial E \cap B_\rho(y)|,
	\end{split}
	\end{equation*}
	where, with no risk of confussion, $|\cdot|$ refers to both the $n$ and $n-1$ dimensional Hausdorff 
	measures. This inequality and \eqref{lemma1} give
	\[
	\frac{d}{d\rho}\Big(-\rho^{-n}\lvert E_\rho\rvert\Big)\leq\rho^{-n}\Big(|\partial E \cap B_\rho(y)|+\int_{E_\rho}\lvert H\rvert\,dV\Big),
	\]
	which is equivalent to
	\begin{equation*}
	\frac{d}{d\rho}\bigg(\rho^{-n}\lvert E_\rho\rvert\,\text{exp}\int_{0}^{\rho}\frac{|\partial E \cap 
	B_\sigma(y)|+\int_{E_\sigma}\rvert H\lvert\,dV}{\lvert E_\sigma \rvert}\,d\sigma\bigg)\geq0.
	\end{equation*}
	Since this holds for almost every $\rho>0$ and the function between parentheses is continuous in $\rho$,
	it follows that it is monotone nondecreasing in $\rho$, and hence
	\begin{equation*}
	\rho^{-n}\lvert E_\rho\rvert\,\text{exp}\int_{0}^{\rho}\frac{|\partial E \cap B_\sigma(y)|+\int_{E_\sigma}\rvert H\lvert\,dV}{\lvert E_\sigma \rvert}\,d\sigma\geq\lim_{\rho\to0}\rho^{-n}\lvert E_\rho\rvert=\omega_n,
	\end{equation*}
	where $\omega_n$ is the volume of the unit ball of $\R^n$.
	
	Next, by choosing $\rho_0:=\big(2\lvert E\rvert\omega_n^{-1}\big)^\frac{1}{n}$, we deduce that
	\begin{equation*}
	\text{exp}\int_{0}^{\rho_0}\frac{|\partial E \cap B_\sigma(y)|+\int_{E_\sigma}\rvert H\lvert\,dV}{\lvert E_\sigma \rvert}\,d\sigma\geq\rho_0^n\omega_n\lvert E_{\rho_0}\rvert^{-1}\geq\rho_0^n\omega_n\lvert E\rvert^{-1}=2.
	\end{equation*}
	Therefore, for every point $y\in E$, there exists a radius $r(y)\in(0,\rho_0)$	such that
	\begin{equation*}
	\rho_0\Big(|\partial E \cap B_{r(y)}(y)|+\int_{E_{r(y)}}\lvert H\rvert\,dV\Big)\geq \lvert E_{r(y)} \rvert \log2.
	\end{equation*}
	If we substitute the chosen value for $\rho_0$, we find
	\begin{equation}
	\label{lemma2}
	\lvert E_{r(y)} \rvert \leq C\lvert E\rvert^\frac{1}{n}\left(|\partial E \cap B_{r(y)}(y)|
	+\int_{E_{r(y)}}\lvert H\rvert\,dV\right),
	\end{equation}
	for some constant $C$ depending only on the dimension $n$. Note that the first term on the right-hand
	side of this inequality is simply the measure of $\partial E$ within the ball, while the corresponding term
	in the starting inequality \eqref{lemma1} counted in addition the measure within 
	$E$ of the boundary of the ball. 
	
	Now, since $y\in E$ is arbitrary, we have that every point in the set $E$ is the center of a ball $B(y)=B_{r(y)}(y)$ for which \eqref{lemma2} holds. Since the union of these balls covers~$E$, the Besicovitch covering theorem gives the existence of a countable sub-collection of balls $\{B(y_i)\}_i$, with the same radii $r(y_i)$ as before, such that 
	\[
	E\subset \bigcup B(y_i)
	\]
	and such that every point in $E$ belongs at most to $N_n$ of the balls $B(y_i)$, where $N_n$ is a constant depending only on $n$. Combining this covering argument with~\eqref{lemma2}, we conclude~\eqref{MSineq}.
\end{proof}

Now, it is standard to deduce the Sobolev inequality of Theorem~\ref{thm_MS_functions} from the isoperimetric inequality~\eqref{MSineq}. 

\begin{proof}[Proof of Theorem~\ref{thm_MS_functions}]
	\textit{Step 1.} First, we prove that for every smooth $\varphi$ it holds that
	\begin{equation}\label{ms_functions}
	\left(\int_M\abs{\varphi}^\frac{n}{n-1}\,dV\right)^\frac{n-1}{n}\leq C\int_M\big(\abs{\nabla_T\varphi}+\abs{H\varphi}\big)\,dV,
	\end{equation}
	where $C$ is a positive constant depending only on $n$.
	
	Let $\mu$ be the measure on $M$ defined by $ d\mu= \abs{\varphi}^\frac{1}{n-1}dV$. Then, by Cavalieri's principle it holds that
	\begin{equation}
	\label{ineq_67}
	\begin{split}
	\hspace{-0.3cm}\int_M\abs{\varphi}^\frac{n}{n-1}\,dV
	&=\int_M\abs{\varphi}\,d\mu
	=\int_{0}^{+\infty}\mu\left(\{\abs{\varphi}>t\}\right)\,dt=\int_{0}^{+\infty}\int_{\{\abs{\varphi}>t\}}\abs{\varphi}^\frac{1}{n-1}\,dV\,dt 
	\\
	&\leq \int_{0}^{+\infty}\left(\int_{\{\abs{\varphi}>t\}}\abs{\varphi}^\frac{n}{n-1}\,dV\right)^\frac1n \abs{\{\abs{\varphi}>t\}}^\frac{n-1}{n}\,dt 
	\\
	&\leq \left(\int_M\abs{\varphi}^\frac{n}{n-1}\,dV\right)^\frac1n \int_{0}^{+\infty}\abs{\{\abs{\varphi}>t\}}^\frac{n-1}{n}\,dt,
	\end{split}
	\end{equation}
	where we used Hölder's inequality in the second line.
	
	From the regularity of $\varphi$ and Sard's theorem, we have that the set of singular values of $\varphi$ has zero Lebesgue measure. Considering only regular values $ t $ in the last line of~\eqref{ineq_67}, we can apply Theorem~\ref{thm_MS&A} to the set $ E= \{\abs{\varphi}>t\} $. In this way, we obtain 
	\begin{equation}\label{ineq_3}
	\begin{split}
	\left(\int_M\abs{\varphi}^\frac{n}{n-1}\,dV\right)^\frac{n-1}{n}&\leq\int_{0}^{+\infty}\abs{\{\abs{\varphi}>t\}}^\frac{n-1}{n}\,dt \\
	&\leq C\left(\int_{0}^{+\infty}\abs{\{\abs{\varphi}=t\}}\,dt+\int_{0}^{+\infty}\int_{\{\abs{\varphi}>t\}}\abs{H}\,dV\,dt
	\right).
	\end{split}
	\end{equation}
	
	Now, in the first integral in the right-hand side of \eqref{ineq_3} we use the coarea formula on manifolds --- see~\cite[Theorem VIII.3.3.]{Ch} --- to write
	\[
	\int_{0}^{+\infty}\abs{\{\abs{\varphi}=t\}}\,dt=\int_M \abs{\nabla_T \varphi}\,dV.
	\]
	Finally, plugging this identity in \eqref{ineq_3} and applying Fubini's Theorem on the last integral in \eqref{ineq_3}, we obtain~\eqref{ms_functions}.
	
	\textit{Step 2.} We can easily extend \eqref{ms_functions} to the case of an exponent $ p\in[1,n) $, proving~\eqref{ms&a}.
	In order to do this, we define $ \psi=\abs{\varphi}^{s-1}\varphi $, with $ s=p^*/1^* $, and we apply~\eqref{ms_functions} to $\psi$. We obtain
	\[
	\left(\int_M\abs{\varphi}^\frac{ns}{n-1}\,dV\right)^\frac{n-1}{n}\leq C\int_M \abs{\varphi}^{s-1}\left(s\abs{\nabla_T\varphi}+\abs{H\varphi}\right)\,dV.
	\]
	Now, exploiting that $ns/(n-1)=1^*s=p^*$, using a Hölder inequality in the right-hand side with exponents~$p$ and~$p'$, and taking into account that $(s-1)p'=p^*$, we get
	\begin{equation*}
	\hspace{-0.5cm}\bigg(\int_M\abs{\varphi}^{p^*}\,dV\bigg)^\frac{n-1}{n}
	\leq C\left(\int_M\abs{\varphi}^{p^*}\,dV\right)^\frac{p-1}{p}\left(\int_M\big(s\abs{\nabla_T\varphi}+\abs{H\varphi}\big)^p\,dV\right)^\frac1p.
	\end{equation*}
	This establishes Theorem~\ref{thm_MS_functions}.
\end{proof}
\section{Hardy inequalities on hypersurfaces}\label{section_hardy}
In this section we establish the two Hardy inequalities on hypersurfaces of~$\R^{n+1}$ stated in Theorems~\ref{thm_hardy_ibp_p2} and~\ref{thm_hardy_gs_p2}. For the first one, we also prove a general version with exponent $p\geq1$, which is stated in Theorem~\ref{thm_hardy_ibp} below.

\subsection{Hardy inequality through integration by parts}
In this subsection we prove the following Hardy inequality, which is the version of Theorem~\ref{thm_hardy_ibp_p2} for a general exponent $p\geq1$.
\begin{theorem} 
	\label{thm_hardy_ibp}
	Let $M$ be a smooth hypersurface of $\R^{n+1}$, $p\geq1$, and $ a\in[0,n) $. Then, for every $\varphi\in C^1_c(M) $ we have
	\begin{equation}\label{hardy_ibp_p}		
	\begin{split}
	&(n-a)\int_M\frac{\abs{\varphi}^{p}}{\abs{x}^a}\,dV+a\int_M\left(\frac{x}{\abs{x}}\cdot\nu_M\right)^2\frac{\abs{\varphi}^{p}}{\abs{x}^a}\,dV
	\\
	&\hspace{3cm}\leq
	\left(\int_M\frac{\abs{\varphi}^{p}}{\abs{x}^a}\,dV\right)^\frac{p-1}{p}\left( \int_M\frac{\lvert p\nabla_T\varphi- \mathcal{H}\varphi\rvert^p}{\abs{x}^{a-p}}\,dV\right)^\frac1p.
	\end{split}
	\end{equation}
\end{theorem}
By throwing the second term in the left-hand side of~\eqref{hardy_ibp_p} and taking $p=a<n$,
we deduce that the Hardy inequality in its Euclidean form and with its best constant, 
	\begin{equation}\label{p_hardy_flat}
	\frac{(n-p)^p}{p^p}\int_{M}\frac{\abs{\varphi}^p}{\abs{x}^p}\,dV
	\leq
	\int_{M} \abs{\nabla\varphi}^p\,dV,
	\end{equation}
	holds on every minimal hypersurface $M$ for all $p\in[1,n)$. As mentioned in our comments 
	following~\eqref{carron}, this inequality also follows from a result in~\cite{KO1}.

	We recall that, when $M=\R^n$, for $1<p<n$ the optimal constant in \eqref{p_hardy_flat} is not 
	achieved by any function 
	in the homogeneous Sobolev space $\dot{W}^{1,p}(\R^n)$ --- the completion of $C^1_c(\R^n)$ 
	with respect to the 
	right-hand side of \eqref{p_hardy_flat}; see \cite{FS}. On the contrary, if $p=1$, every 
	radially symmetric decreasing function realizes the equality in \eqref{p_hardy_flat}
	--- as it can be checked using the coarea formula, the layer cake 
	representation for the function $\varphi$, and the fact that ${\rm div }(x/|x|)=(n-1)/|x|$.

\begin{proof}[Proof of Theorem~\ref{thm_hardy_ibp}.]
	Using formula \eqref{div2} for the tangential divergence of the position vector $x$, and then integrating by parts according to~\eqref{ibp_vector}, we can write
	\begin{equation*}
	\begin{split}
	&n\int_M\frac{\abs{\varphi}^p}{\abs{x}^a}\,dV=\int_M\frac{\abs{\varphi}^p}{\abs{x}^a}\,\textnormal{div}_Tx\,dV 
	\\
	&\hspace{1.2cm}=-\int_M\left(p\frac{\abs{\varphi}^{p-2}\varphi}{\abs{x}^a}\nabla_T\varphi\cdot x+ \abs{\varphi}^px\cdot\nabla_T\abs{x}^{-a} -	\frac{\abs{\varphi}^p}{\abs{x}^a}\mathcal{H}\cdot x\right)\,dV.
	\end{split}
	\end{equation*}
	Now, recalling that the tangential part of the position vector $x$ is $x_T=x-(x\cdot\nu_M)\nu_M$, 
	we compute
	\begin{equation*}
	\begin{split}
	x\cdot\nabla_T\abs{x}^{-a}
	=-a\abs{x}^{-a-2}x\cdot x_T
	=-a\abs{x}^{-a}+a\left(\frac{x}{\abs{x}}\cdot\nu_M\right)^2\abs{x}^{-a}.
	\end{split}
	\end{equation*}
	Hence, we have
	\begin{equation}\label{ineq_89}
	\begin{split}
	&(n-a)\int_M\frac{\abs{\varphi}^p}{\abs{x}^a}\,dV+a\int_M\left(\frac{x}{\abs{x}}\cdot\nu_M\right)^2\frac{\abs{\varphi}^p}{\abs{x}^a}\,dV
	\\
	&\hspace{3.5cm}=-\int_M\frac{\abs{\varphi}^{p-2}\varphi}{\abs{x}^{a-1}}\left(p\nabla_T\varphi\cdot \frac{x}{\abs{x}} -\varphi\mathcal{H}\cdot \frac{x}{\abs{x}}\right)\,dV
	\\
	&\hspace{3.5cm}
	\leq
	\int_M\frac{\abs{\varphi}^{p-1}}{\abs{x}^{a-1}}\abs{p\nabla_T\varphi-\mathcal{H}\varphi}\,dV.
	\end{split}
	\end{equation}
	
	Finally, we apply Hölder's inequality with exponents $ p $ and $ p' $ to the last integral in~\eqref{ineq_89}, obtaining
	\begin{equation*}
	\begin{split}
	&\int_M\frac{\abs{\varphi}^{p-1}}{\abs{x}^{a-1}}\abs{p\nabla_T\varphi-\mathcal{H}\varphi}\,dV=\int_M\frac{\abs{\varphi}^{p-1}}{\abs{x}^{a(p-1)/p}}\frac{\abs{p\nabla_T\varphi-\mathcal{H}\varphi}}{\abs{x}^{(a-p)/p}}\,dV
	\\
	&\hspace{4cm}\leq\left(\int_M\frac{\abs{\varphi}^p}{\abs{x}^a}\,dV\right)^\frac{p-1}{p}
	\left(\int_M\frac{\abs{p\nabla_T\varphi-\mathcal{H}\varphi}^p}{\abs{x}^{a-p}}\,dV	\right)^\frac1p.
	\end{split}
	\end{equation*}
	Plugging this bound in \eqref{ineq_89}, we obtain \eqref{hardy_ibp_p} and finish the proof of Theorem~\ref{thm_hardy_ibp}.
\end{proof}

When $p=2$ and $n\geq3$, we exploit a nice simplification in \eqref{hardy_ibp_p} and prove Theorem~\ref{thm_hardy_ibp_p2}.

\begin{proof}[Proof of Theorem~\ref{thm_hardy_ibp_p2}.]
We use \eqref{hardy_ibp_p} with $ p=2 $. Then, since the vectors $\nabla_T\varphi$ and $\mathcal{H}$ are orthogonal, we have
	\[
	\abs{2\nabla_T\varphi-\mathcal{H}\varphi}^2=4\abs{\nabla_T \varphi}^2+\abs{\mathcal{H}\varphi}^2
	\]
	and Theorem~\ref{thm_hardy_ibp_p2} follows directly from Theorem~\ref{thm_hardy_ibp}.
\end{proof}

From Theorem~\ref{thm_hardy_ibp} we deduce a version with exponent $p$ for the foliated Hardy inequality~\eqref{hardy_cabre} that the first author established for $p=2$ in~\cite{C1}. 
In the statement, we use the following notation for the radial derivative:
\[
u_{r}=\nabla u \cdot \frac{x}{\abs{x}}.
\]
Recall that the mean curvature $\mathcal{H}$ and the tangential gradient $\nabla_T$ refer to the level sets of 
the function $u$. The result is the following.

\begin{corollary}\label{cor_hardy}
	Let $\Omega$ be a smooth bounded domain of $\R^{n+1} $, $u$ a $C^\infty(\overline{\Omega})$ function, $p\geq1$, and $a\in[0,n)$. Then, for every $\varphi\in C^1_c(\Omega) $ we have
	\begin{equation}\label{hardy_levelsets}
	\begin{split}
	&(n-a)\int_{\Omega}\abs{\nabla u}\frac{\abs{\varphi}^p}{\abs{x}^a}\,dx
	+a\int_{\Omega}\frac{u_{r}^2}{\abs{\nabla u}}\,\frac{\abs{\varphi}^p}{\abs{x}^a}\,dx
	\\
	&\hspace{2.3cm}\leq
	\left(\int_{\Omega}\abs{\nabla u}\frac{\abs{\varphi}^p}{\abs{x}^a}\,dx\right)^\frac{p-1}{p}
	\left(\int_{\Omega}\abs{\nabla u}\frac{\lvert p\nabla_T \varphi-\mathcal{H}\varphi\rvert^p}{\abs{x}^{a-p}}\,dx\right)^\frac1p.
	\end{split}
	\end{equation}
\end{corollary}

\begin{proof}
	Using the coarea formula in Euclidean space for the two integrals in the left-hand side of \eqref{hardy_levelsets}, we see that
	\begin{equation}\label{3_67}
	\begin{split}
	&(n-a)\int_{\Omega}\abs{\nabla u}\frac{\abs{\varphi}^p}{\abs{x}^a}\,dx
	+a\int_{\Omega}\frac{u_{r}^2}{\abs{\nabla u}}\,\frac{\abs{\varphi}^p}{\abs{x}^a}\,dx
	\\
	&\hspace{1cm}=(n-a)\int_{\R}\int_{\{u=t\}}\frac{\abs{\varphi}^p}{\abs{x}^a}\,dV\,dt
	+a\int_{\R}\int_{\{u=t\}}
	\left(\frac{x}{\abs{x}}\cdot\frac{\nabla u}{\abs{\nabla u}}
	\right)^2\frac{\abs{\varphi}^p}{\abs{x}^a}\,dV\,dt.
	\end{split}
	\end{equation}
	
	Now, by Sard's theorem, $\{u=t\}$ is a smooth hypersurface of~$\R^{n+1}$ for almost every $t\in\R$, and the normal vector $\nu_M$ of $ M=\{u=t\} $ is
	\[
	\nu_M=\frac{\nabla u}{\abs{\nabla u}}.
	\]
	Therefore, we can apply \eqref{hardy_ibp_p} to the function~$\varphi$ on each smooth hypersurface $M=\{u=t\}$ and then integrate in $dt$, obtaining
	\begin{equation*}
	\begin{split}
	&(n-a)\int_{\R}\int_{\{u=t\}}\frac{\abs{\varphi}^p}{\abs{x}^a}\,dV\,dt+a\int_{\R}\int_{\{u=t\}}\left(\frac{x}{\abs{x}}\cdot\frac{\nabla u}{\abs{\nabla u}}
	\right)^2\frac{\abs{\varphi}^p}{\abs{x}^a}\,dV\,dt
	\\
	&\hspace{1.5cm}\leq\left(\int_{\R}\int_{\{u=t\}}\frac{\abs{\varphi}^p}{\abs{x}^a}\,dV\,dt\right)^\frac{p-1}{p}
	\left(\int_{\R}\int_{\{u=t\}}\frac{\lvert p\nabla_T\varphi-\mathcal{H}\varphi\rvert^p}{\abs{x}^{a-p}}\,dV\,dt\right)^\frac1p,
	\end{split}
	\end{equation*}
	where we have used Hölder's inequality for an integral in $dt$.
	Finally, using again the coarea formula and combining this inequality with \eqref{3_67}, we deduce \eqref{hardy_levelsets}.
\end{proof}

\subsection{Hardy inequality through a ground state substitution}
In this subsection we prove Theorem~\ref{thm_hardy_gs_p2} using a method known as the ground state substitution. 
Within the proof we will need that
	\begin{equation}\label{divT}
	\begin{split}
	\textnormal{div}_T x_T &={\rm div}_T\big(x-(x\cdot\nu_M)\nu_M\big)
	= n-(x\cdot\nu_M){\rm div}_T\nu_M - \big(\nabla_T (x\cdot\nu_M)\big)\cdot \nu_M
	\\&= n-(x\cdot\nu_M)H,
	\end{split}
	\end{equation}
	where we have used that $\textnormal{div}_Tx=n$, by \eqref{div2}, and that $\textnormal{div}_T\nu_M=H$.

	It is now easy to deduce the inequality 
	$$
	\Delta |x|\geq \frac{n-1}{|x|} - \left( \frac{x}{|x|}\cdot \nu_M\right) H
	$$
	for the Laplace-Beltrami operator on $M$ --- a result mentioned in the Introduction within the 
	context of minimal hypersurfaces. Indeed, we have
	\begin{equation*}
	\begin{split}
	\Delta |x| &={\rm div}_T \nabla_T |x|={\rm div}_T (x_T/|x|) 
	=({\rm div}_T x_T)/|x|+x_T\cdot\nabla_T |x|^{-1}
	\\&= \left( n-(x\cdot\nu_M)H\right) /|x|-|x|^{-3}|x_T|^2 
	\\&\geq (n-1)/|x| -(x\cdot\nu_M)H /|x|, 
	\end{split}
	\end{equation*}
	as claimed.

\begin{proof}[Proof of Theorem~\ref{thm_hardy_gs_p2}]
	We substitute $\varphi(x)=\omega(x) v(x)$, with $ \omega(x)=\lvert x\rvert^{-\frac{n-2}{2}}$ and $v\in C^1_c(M)$, in the gradient term
	\begin{equation}
	\label{gs0}
	\int_M \lvert\nabla_T\varphi\rvert^2\,dV=\int_M\lvert v\nabla_T\omega+\omega\nabla_Tv\rvert^2\,dV.
	\end{equation}
	Applying the convexity inequality $\abs{a+b}^2\geq\abs{a}^2+2\,a\cdot b$, valid for all vectors $ a,b\in \R^n $, we obtain
	\begin{equation*}	
	\int_M \lvert\nabla_T\varphi\rvert^2\,dV\geq\int_M v^2\abs{\nabla_T\omega}^2 \,dV
	+\int_M \omega\,\nabla_T\omega \cdot\nabla_T\left(v^2\right)\,dV.
	\end{equation*}
	Using the formula of integration by parts \eqref{ibp_vector}, we get
	\begin{equation}
	\label{gs3}
	\begin{split}
	&\hspace{-1cm}\int_M \lvert\nabla_T\varphi\rvert^2\,dV\geq 
	\int_M v^2\abs{\nabla_T\omega}^2\,dV-\int_M v^2\textnormal{div}_T\left(\omega\nabla_T\omega\right)\,dV
	\\
	&\hspace{6.5cm}+\int_M\omega\,v^2\nabla_T\omega\cdot\mathcal{H}\,dV.
	\end{split}
	\end{equation}
	Since $ \nabla_T\omega $ is a tangent vector and the mean curvature vector $ \mathcal{H} $ is normal to~$M$, the last term in \eqref{gs3} vanishes. Exploiting an additional cancellation after developing the divergence in \eqref{gs3}, we have
	\begin{equation}
	\label{gs1}
	\int_M \lvert\nabla_T\varphi\rvert^2\,dV\geq-\int_M \omega\,v^2\textnormal{div}_T\left(\nabla_T\omega\right)\,dV.
	\end{equation}
	
	Next, we compute the tangential divergence of the vector field $ \nabla_T\omega $, where $ \omega(x)=\lvert x\rvert^\alpha$ with $\alpha=-(n-2)/2$. The tangential gradient of $\omega$ is
	\[\nabla_T\omega=\alpha\abs{x}^{\alpha-2}x_T
	=\alpha\abs{x}^{\alpha-2}\big(x-(x\cdot\nu_M)\nu_M\big). \]
	Hence, using \eqref{divT}, we have
	\begin{equation*}
	\begin{split}
	-\textnormal{div}_T\left(\nabla_T\omega\right)&=-\alpha\,\textnormal{div}_T\left(\abs{x}^{\alpha-2}\big(x-(x\cdot\nu_M)\nu_M\big)\right)
	\\&=-\alpha\abs{x}^{\alpha-2}\left(n-x\cdot\nu_MH\right)-\alpha(\alpha-2)\abs{x_T}^{2}\abs{x}^{\alpha-4}.
	\end{split}
	\end{equation*}
	We plug this into \eqref{gs1}, recalling that $\omega(x)=\abs{x}^\alpha$, and obtain
	\begin{equation}\label{gs4}
	\begin{split}
	&\hspace{-0.7cm}\int_M \lvert\nabla_T\varphi\rvert^2\,dV\geq
	\alpha\int_M \abs{x}^{2\alpha-2}v^2
	x\cdot\mathcal{H}\,dV
	\\ 
	&\hspace{1cm}-n\alpha\int_M\abs{x}^{2\alpha-2}v^{2}\,dV
	-\alpha(\alpha-2)\int_M\abs{x_T}^{2}\abs{x}^{2\alpha-4}v^{2}\,dV.
	\end{split}
	\end{equation}
	
	Now we move the first integral in the right-hand side of \eqref{gs4} to the left-hand side of the inequality, and observe that
	$\abs{x}^{2\alpha-2}v^2=\varphi^2/\abs{x}^2$.
	Therefore, \eqref{gs4} reads
	\begin{equation*}
	\begin{split}
	&\hspace{-1cm}\int_M\left(\abs{\nabla_T\varphi}^2+\frac{n-2}{2}\,\frac{\varphi^2}{\abs{x}^{2}}\,x\cdot\mathcal{H}\right)\,dV
	\\ 
	&\hspace{1.25cm}\geq-n\alpha\int_M\frac{\varphi^2}{\abs{x}^{2}}\,dV
	-\alpha(\alpha-2)\int_M\frac{\abs{x_T}^{2}}{\abs{x}^{2}}\,\frac{\varphi^2}{\abs{x}^{2}}\,dV.
	\end{split}
	\end{equation*}
	In the last integral, we have $\abs{x_T}^2=\abs{x}^2-(x\cdot\nu_M)^2$
	and thus the inequality becomes
	\begin{equation*}
	\begin{split}
	&\hspace{-0.5cm}\int_M\left(\abs{\nabla_T\varphi}^2+\frac{n-2}{2}\,\frac{\varphi^2}{\abs{x}^{2}}\,x\cdot\mathcal{H}\right)\,dV
	\\ 
	&\hspace{1cm}\geq-\alpha(n+\alpha-2)\int_M\frac{\varphi^2}{\abs{x}^{2}}\,dV
	+\alpha(\alpha-2)\int_M\left(\frac{x}{\abs{x}}\cdot\nu_M\right)^2\frac{\varphi^2}{\abs{x}^{2}}\,dV.
	\end{split}
	\end{equation*}
	
	Finally, since $-\alpha\left(n+\alpha-2\right)=(n-2)^2/4$ and $\alpha(\alpha-2)=(n^2-4)/4$, we conclude~\eqref{hardy_car}.
\end{proof}

\section{Hardy-Sobolev and Hardy-Poincaré inequalities on hypersurfaces}\label{section_mixed}
In this section we prove the Hardy-Sobolev inequality stated in Corollary~\ref{cor_ckn_p2} and the Hardy-Poincaré inequality of Theorem~\ref{thm_hardy+_p2}.

We start from the Hardy-Sobolev inequality on hypersurfaces, that we obtain as an interpolation of the Michael-Simon and Allard inequality and the Hardy inequality of Theorem~\ref{thm_hardy_ibp}. We state and prove here our result for a general power $p\in[1,n)$.

\begin{corollary}\label{cor_ckn}
	Let $M$ be a smooth hypersurface of $\R^{n+1}$, $ p\in[1,n) $, and $ b\in[0,1] $. Then, for every $ \varphi\in C^1_c(M) $ we have
	\begin{equation}\label{h_s}
	\left(\int_M\frac{\abs{\varphi}^{p\frac{n-bp}{n-p}}}{\abs{x}^{bp}}\,dV\right)^\frac{n-p}{n-bp}\leq C\int_M\big(\abs{\nabla_T\varphi}^p+\abs{H\varphi}^p\big)\,dV,
	\end{equation}
	for some positive constant $C$ depending only on $n$ and $p$.
\end{corollary}

\begin{proof} 
	First, from \eqref{hardy_ibp_p} with $a=p$ it follows that
	\begin{equation*}		
	\begin{split}
	(n-p)\int_M\frac{\abs{\varphi}^{p}}{\abs{x}^p}\,dV&\leq(n-p)\int_M\frac{\abs{\varphi}^{p}}{\abs{x}^p}\,dV+p\int_M\left(\frac{x}{\abs{x}}\cdot\nu_M\right)^2\,\frac{\abs{\varphi}^{p}}{\abs{x}^p}\,dV
	\\
	&\leq
	\left(\int_M\frac{\abs{\varphi}^{p}}{\abs{x}^p}\,dV\right)^\frac{p-1}{p}\left( \int_M\lvert p\nabla_T\varphi- \mathcal{H}\varphi\rvert^p\,dV\right)^\frac1p.
	\end{split}
	\end{equation*}
	Raising the inequality to the power $p$ and using the convexity inequality $\abs{a+b}^p\leq 2^{p-1}\left(\abs{a}^p+\abs{b}^p\right)$,
	we obtain
	\begin{equation}\label{basic_hardy}
	\begin{split}
	(n-p)^p\int_M\frac{\abs{\varphi}^{p}}{\abs{x}^p}\,dV
	&\leq
	\int_M \abs{p\nabla_T\varphi-\mathcal{H}\varphi}^p\,dV
	\\
	&\leq
	2^{p-1}\int_M \left(p^p\abs{\nabla_T\varphi}^p+\abs{H\varphi}^p\right)\,dV.
	\end{split}
	\end{equation}
	
	Observe that, if $b=0$ or $b=1$, then \eqref{h_s} follows respectively from the Michael-Simon and Allard inequality~\eqref{ms&a} or from the Hardy inequality~\eqref{basic_hardy}. Thus, we can assume $b\in(0,1)$ in the rest of the proof.
	
	Now, we consider the integral in the left-hand side of \eqref{h_s}. Using Hölder's inequality with exponents $ 1/b $ and $ 1/(1-b) $, the Hardy inequality~\eqref{basic_hardy}, and Theorem~\ref{thm_MS_functions}, we get
	\begin{equation*}
	\begin{split}
	\int_M\frac{\abs{\varphi}^{p\frac{n-bp}{n-p}}}{\abs{x}^{bp}}\,dV&=\int_M\left(\frac{\abs{\varphi}}{\abs{x}}\right)^{bp}\abs{\varphi}^{(1-b)\frac{np}{n-p}}\,dV
	\\
	&\leq\left(\int_M\frac{\abs{\varphi}^{p}}{\abs{x}^{p}}\,dV\right)^b
	\left(\int_M\abs{\varphi}^{p^*}dV\right)^{1-b}
	\\
	&\leq C\left(\int_M\big(\abs{\nabla_T\varphi}^p+\abs{H\varphi}^p\big)\,dV\right)^\beta,
	\end{split}
	\end{equation*}
	where $C$ is a positive constant depending only on $n$ and $p$, while $\beta$ is
	\[\beta=b+\frac{(1-b)p^*}{p}=\frac{n-bp}{n-p}.\]
	Finally, raising the inequality to the power $1/\beta$, \eqref{h_s} is established. Observe that, since $\beta>1$, $C^{1/\beta}\leq C$ if we take $C\geq1$. Hence, the final constant depends only on~$n$ and~$p$.
\end{proof}

The remaining part of this section is devoted to the proof an improved Hardy inequality in the Poincaré sense, stated in Theorem~\ref{thm_hardy+_p2}. 
Its proof is based on a modification of the ground state substitution method, that we have used in Theorem~\ref{thm_hardy_gs_p2}, and on a Poincaré inequality with weights stated next.

The following is a Poincaré inequality with exponent $p\geq1$ and a weight of the type $ \abs{x}^{-a} $, for functions with compact support on a hypersurface $M$ (more precisely, with support in a ball of radius $r$). 

\begin{proposition}
	\label{prop_poincare}
	Let $M$ be a smooth hypersurface of $\R^{n+1}$, $B_r=B_r(0)\subset\R^{n+1}$ the open ball of radius $r$ centered at the origin, $p\geq1$, and $ a\in[0,n) $. Then, for every $\varphi\in C^1_c(B_r\cap M) $ we have
	\begin{equation}
	\label{poinc_ineq}
	(n-a)^p\int_M\frac{\abs{\varphi}^p}{\abs{x}^a}\,dV\leq 2^{p-1}r^p\int_M \left(p^p\frac{\abs{\nabla_T \varphi}^p}{\abs{x}^a}+\frac{\abs{H\varphi}^p}{\abs{x}^a}\right)\,dV.
	\end{equation}
\end{proposition}

\begin{proof}
	As in the proof of Corollary~\ref{cor_ckn}, but with $a\in[0,n)$ instead of~$a=p$, from~\eqref{hardy_ibp_p} we obtain
	\begin{equation*}
	(n-a)^p\int_M\frac{\abs{\varphi}^{p}}{\abs{x}^a}\,dV
	\leq 2^{p-1}
	\int_M \frac{p^p\abs{\nabla_T\varphi}^p+ \abs{H\varphi}^p}{\abs{x}^{a-p}}\,dV.
	\end{equation*}
	Then, taking advantage of the fact that the support of $\varphi$ is contained in $B_r(0)$, we can bound $\abs{x}^p\leq r^p$ and obtain \eqref{poinc_ineq}.
\end{proof}

Now, we can prove Theorem \ref{thm_hardy+_p2}. Note that here we assume $p=2$ and $n\geq2$.

\begin{proof}[Proof of Theorem~\ref{thm_hardy+_p2}]
	As in the proof of Theorem~\ref{thm_hardy_gs_p2}, we use the ground state substitution $\varphi=v\omega $, where $ \omega(x)=\abs{x}^{-(n-2)/2}$. 
	We proceed as in the proof of Theorem~\ref{thm_hardy_gs_p2}, but in the right-hand side of \eqref{gs0} we use the 
	identity\footnote{
	For an exponent $p\neq2$, here one would use a well-known convexity inequality instead of this identity (see Lemma~2.6 and Remark~2.7 in~\cite{FS}, or~\cite[Lemma 4.2]{Lind}).}
	$\abs{a+b}^2=\abs{a}^2+2\,a\cdot b+\abs{b}^2$ for vectors $a,b\in\R^n$.
	Therefore, we find
	\begin{equation}
	\label{hardy+}
	\begin{split}
	&\hspace{-1.5cm}\int_M \left(\abs{\nabla_T\varphi}^2+ \frac{n-2}{2}\,\frac{\abs H\varphi^2}{\abs{x}}\right)\,dV
	\geq\frac{(n-2)^2}{4}\int_M\frac{\varphi^2}{\abs{x}^2}\,dV
	\\
	&\hspace{1.75cm}+\frac{n^2-4}{4}
	\int_M\left(\frac{x}{\abs{x}}\cdot\nu_M\right)^2\frac{\varphi^2}{\abs{x}^2}\,dV
	+\int_M\frac{\abs{\nabla_Tv}^2}{\abs{x}^{n-2}}\,dV.
	\end{split}
	\end{equation}
	
	Next, to control the last integral in \eqref{hardy+} from below, we use inequality~\eqref{poinc_ineq} with~$\varphi=v$, $p=2$, and~$a=n-2$. Observe that this forces $n\geq2$. In this way, we have
	\begin{equation}\label{poinc+}
	\int_M\frac{\abs{\nabla_Tv}^2}{\abs{x}^{n-2}}\,dV \geq \frac{1}{2r^2}\int_M\frac{v^2}{\abs{x}^{n-2}}\,dV-\frac{1}{4}\int_{M}\frac{\abs{Hv}^2}{\abs{x}^{n-2}}\,dV.
	\end{equation}
	Finally, combining \eqref{hardy+} and \eqref{poinc+}, and using the fact that $v^2/\abs{x}^{n-2}=\varphi^2$, \eqref{imprhardy_p2} is established.
\end{proof}

\begin{appendices}
	\section{Notation for tangential derivatives}\label{appendix_giusti}
	In the setting of hypersurfaces of Euclidean space, tangential derivatives can be defined in an elementary calculus way without using Riemannian geometry, for instance as presented in Giusti's book~\cite{G}. Throughout the paper, we adopt this definition of tangential derivatives, that we recall next. From it, one can define the tangential divergence of a vector field. Alternatively, one can define the tangential divergence intrinsically using Riemannian geometry, as done for instance in~\cite{Ch}. In this appendix, and for completeness, we introduce and compare these two notions in the setting of hypersurfaces of~$\R^{n+1}$.
	We start by giving the former definition, following~\cite{G}.
	\begin{definition}
		\label{def_div}
		Let $M$ be a smooth hypersurface of $\R^{n+1}$ with normal vector $\nu_M$.
		\begin{itemize}
			\item[(a)]
			Let $\varphi$ be a $C^1$ function defined on $M$.
			We define the $i-$th tangential derivative of $\varphi$, for $i=1,\dots,n+1$, as
			\[
			\delta_i\varphi:=\partial_i\varphi-\nu_M^i\sum_{j=1}^{n+1}(\partial_j\varphi)\nu_M^j,
			\]
			where $\nu_M^j$ is the $j-$th component of the normal vector $\nu_M$ to $M$ and $\partial_j\varphi$ is the $j-$th partial derivative of $\varphi$, once the function $\varphi$ has been extended to all of~$\R^{n+1}$.
			\item[(b)] 
			With $\varphi$ as in (a),
			we define the tangential gradient of $\varphi$ as the vector
			\[
			\nabla_T\varphi=\nabla\varphi-(\nabla\varphi\cdot\nu_M)\nu_M=(\delta_1\varphi,\delta_2\varphi,\dots,\delta_{n+1}\varphi).
			\]
			Note that $ \nabla_T\varphi\cdot\nu_M=0 $ for every $C^1$ function $\varphi$ defined on $M$.
			\item[(c)]	Let $Z$ be a $C^1$ vector field defined on $M$ with values in $\R^{n+1}$, not necessarily tangent to $M$, and whose components are $ Z^i $ with $ i=1,\dots,n+1 $. We define its tangential divergence as
			\begin{equation}\label{divT_def}
			\textnormal{div}_T Z= \sum_{i=1}^{n+1}\delta_i Z^i.
			\end{equation}
		
		\end{itemize}
	\end{definition}
	From the definitions, it easily follows that
			\begin{equation}
			\label{div1}
			\textnormal{div}_T(\varphi Z)=\nabla_T\varphi\cdot Z+\varphi\,\textnormal{div}_T Z.
			\end{equation}
			
	Observe that this definition of tangential derivatives is extrinsic and it does not give a basis of the $n$-dimensional tangent space of $M$, as the tangential derivatives~$\delta_i$ for $i=1,\dots,n+1$ are linearly dependent. However, if one is familiar with Riemannian geometry, then it is possible to check that, in the case of hypersurfaces of $\R^{n+1}$, the intrinsic Riemannian notion of divergence coincides with $\textnormal{div}_T$ defined in \eqref{divT_def}.
	We recall that the divergence of a \textit{tangent} vector field $Y$ on a general Riemannian manifold $(M,g)$ 
	is defined in an intrinsic way as
	\begin{equation}\label{intrinsic_div}
	\textnormal{div}Y=\textnormal{tr}\big(\xi\longmapsto\nabla_\xi Y\big),
	\end{equation}
	where $ \nabla $ is the Levi-Civita connection of $ (M,g) $. 
	Now, Proposition~II.2.1 in~\cite{Ch} states that, given two Riemannian manifolds $(M, g)$ and $(\overline{M}, \overline{g})$ with $M$ isometrically embedded in $ \overline{M} $ and whose Levi-Civita connections are $\nabla$ and $\overline{\nabla}$, then for every $p\in M$, $\xi\in T_pM$, and vector field $Y\in TM$ on $M$, we have that
	\begin{equation*}
	\nabla_\xi Y=(\overline{\nabla}_\xi Y)_T,
	\end{equation*}
	where $ (\overline{\nabla}_\xi Y)_T $ denotes the tangential component of $ \overline{\nabla}_\xi Y $ with respect to $ M $.
	Therefore, if $\overline{M}=\R^{n+1}$, $M$ is an isometrically embedded hypersurface of $\R^{n+1}$, and $Y$ is a tangent vector field on $M$, then we have 
	\begin{equation*}
	\text{div}Y=
	\textnormal{tr}\big(\xi\longmapsto\nabla_\xi Y\big)=\textnormal{tr}\left(\xi\longmapsto\left(\overline{\nabla}_\xi Y\right)_T\right)
	=\sum_{i=1}^{n+1}\delta_iY^i=\textnormal{div}_T Y,
	\end{equation*}
	where $\text{div}$ is defined in \eqref{intrinsic_div} and $\text{div}_T$ in~\eqref{divT_def}.

	Next, adopting the notion of tangential derivatives from Definition \ref{def_div}, we report a formula of integration by parts proved in~\cite{G}. For all $C^1$ functions $v$ and $w$ such that at least one of them has compact support on $M$, we have that 
	\begin{equation}\label{ibp_functions}
	\int_M\left(\delta_i v\right)w\,dV=-\int_Mv\left(\delta_iw \right)\,dV+\int_MvwH\nu_M^i\,dV,
	\end{equation}
	where $i\in\{1,\dots,n+1\}$, $\nu_M$ is the normal vector to $M$, and $H$ is the mean curvature of $M$.
	For the proof of~\eqref{ibp_functions} we refer to\footnote{
		We point out two typos in~\cite[Lemma~10.8]{G}: first, the mean curvature $H$ is missing in the statement, but not in the proof; second, there is a sign error in front of the integral in the right-hand side, both in the statement and in the proof. The correct statement is~\eqref{ibp_functions}.} 
	\cite[Lemma 10.8]{G} or to~\cite[Lemma 2.1]{C1}. If instead we consider a $C^1$ function $v$ and a $C^1$ vector field $Z$, such that at least one of them has compact support on~$M$, then from \eqref{ibp_functions} we easily deduce		
	\begin{equation}\label{ibp_vector}
	\int_M v\,\textnormal{div}_T Z\,dV
	=-\int_M \nabla_T v\cdot Z\,dV
	+\int_M v Z\cdot\mathcal{H}	\,dV,
	\end{equation}
	where $\mathcal{H}=H\nu_M$ is the mean curvature vector of $M$.
	Indeed, to show \eqref{ibp_vector} it is sufficient to write $\textnormal{div}_TZ=\sum_{i=1}^{n+1}\delta_i Z^i$ and apply \eqref{ibp_functions} on every term of the sum.
	
	Observe that, if $Z$ is tangent then the mean curvature term 
	in \eqref{ibp_vector} vanishes --- since $\mathcal{H}$ is normal to~$M$. 
	
	The following divergence formula with a boundary term is the analogue result to \eqref{ibp_vector}
	with $v\equiv 1$ when $Z$
	does not have compact support.
	Given a $C^1$ {\it tangent} vector field $Z$ defined on~$M$ and a smooth domain $\Omega\subset M$, 
	we have that
	\begin{equation}
	\label{app_divthm}
	\int_\Omega \textnormal{div}_T Z\,dV=\int_{\partial\Omega}Z\cdot \nu_\Omega\,dA,
	\end{equation}
	where $\nu_\Omega\in TM$ is the outward unit normal to $\Omega$. 
	This identity can be proved using a suitable modification of the argument in~\cite[Lemma 10.8]{G}.
	One can also deduce \eqref{app_divthm} from~\cite[Theorem~III.7.5]{Ch}, i.e., the divergence 
	formula on Riemannian manifolds. To this end, one must recall that in~\cite{Ch} the tangential 
	divergence is defined as in \eqref{intrinsic_div} and, in the setting of hypersurfaces of~$\R^{n+1}$, 
	definition \eqref{intrinsic_div} is equivalent to the one we gave in Definition~\ref{def_div}.
	
	\section{Optimal constants in the Michael-Simon and Allard inequality}\label{appendix_constant}
	
	For an integer $k\in [2, n] $, a $k$-dimensional submanifold $M$ of $\R^{n+1}$ with mean curvature $H$, and 
	a smooth domain $E\subset M$ with compact closure in $M$, the Michael-Simon and Allard inequality states that
	\begin{equation}\label{appendix_best_constant}
	\lvert E \rvert^{\frac{k-1}{k}}\leq C_1\textnormal{Per}(E)+C_2\int_E\lvert H \rvert\,dV,
	\end{equation}
	for some positive constants $C_1$ and $C_2$ depending only on $k$.
	Most of the literature on the topic of sharp constants for \eqref{appendix_best_constant}
	is focused on one of two important particular cases: either when the submanifolds 
	$M$ are minimal or when they are compact without boundary and we take $E=M$. The proofs in~\cites{A,MS}
	do not give sharp constants in any of these two situations.
	
	In the former case the mean curvature of $M$ is identically zero, and the problem is finding the optimal 
	constant $C_1$ in the isoperimetric inequality on minimal submanifolds of~$\R^{n+1}$. 
	Under the additional assumption that the submanifold is area minimizing, 
	Almgren~\cite{Alm} proved that the isoperimetric inequality with the Euclidean constant holds, 
	i.e., for every smooth domain $E\subset M$ with compact closure in $M$, one has
	\begin{equation}\label{iso}
	k\, \omega_k^{\frac1k}\abs{E}^{\frac{k-1}{k}}\leq  \textnormal{Per}(E),
	\end{equation}
	where $ \omega_k $ is the volume of the $ k $-dimensional unit ball.
	Back to the general context of non minimizers, in the case of two-dimensional minimal surfaces 
	of $\R^{n+1}$ (i.e., with $ k=2 $) some partial results have been available
	for a good number of years. Leon Simon obtained the desired inequality with half of the expected constant
	\[2\pi\abs{E}\leq  \textnormal{Per}(E)^2.\] 
	He never published the proof of this result, but it can be found in the papers~\cite{Choe,Top}. 
	In~\cite{Top}, Topping improved it to give a simple proof of the Michael-Simon and Allard inequality for 2-dimensional submanifolds of $\R^{n+1}$, not necessarily minimal. 
	The constant $2\pi$ in Simon's inequality on minimal surfaces was improved by Stone~\cite{Stone} 
	(the same improvement is attributed in~\cite{Choe} also to A. Ros), but still without achieving 
	the constant $4\pi$ conjectured in \eqref{iso}. See the survey~\cite{Choe} for a detailed exposition of 
	the problem. 
	Finally, the conjecture for arbitrary dimension~$k$ has 
	been very recently proved by Brendle~\cite{Brendle} in the 
	case of codimension 1 and 2. His method uses a clever extension of the proof 
	of the sharp Euclidean isoperimetric inequality found by the first author in~\cite{C08}. 
	Thus, both proofs use the solution of a Neumann problem, together with the ABP method. 
	In addition, Brendle's proof allows to characterize flat disks as the only cases in which equality
	is achieved.

	The second particular case of \eqref{appendix_best_constant} consists of $M$ being a compact manifold 
	without boundary and $E=M$. 
	Then, inequality~\eqref{appendix_best_constant} reads  
	\begin{equation}\label{aleksandrov}
	\lvert M \rvert^{\frac{k-1}{k}}\leq C_2\int_{M}\lvert H \rvert\,dV
	\end{equation}
	with $2\leq k \leq n$, and the problem of finding the optimal constant $ C_2 $ is still open.
	If~$M=\partial A$ and $A\subset\R^{n+1}$ is a smooth bounded domain which is also assumed to be convex, 
	then~\eqref{aleksandrov} holds with $k=n$ and equality is only achieved when $A$ is a ball, as a consequence of the classical Aleksandrov-Fenchel inequality~\cites{Ale1, Ale2}. 
	More recently, Guan and Li~\cite{GL}, and Huisken and Ilmanen~\cite{HI}, relaxed the convexity assumption 
	with weaker hypothesis on~$A$, obtaining the sharp result in their settings. 
	For a survey on the subject, see~\cite{CW}.
\end{appendices}
	
	\addsection*{Acknowledgments} The authors would like to thank Matteo Cozzi for interesting and useful discussions on the topic of this paper.

\end{document}